\documentclass[11pt,a4paper]{amsart}
\usepackage{a4wide}
\usepackage{amsfonts,amssymb,amsmath,amsthm}

\newcommand{\N}{{\mathbb N}}
\newcommand{\R}{{\mathbb R}}
\renewcommand{\S}{{\mathbb S}}
\newcommand{\eps}{\varepsilon}

\numberwithin{equation}{section}

\newtheorem{theorem}{Theorem}[section]
\newtheorem{lemma}{Lemma}[section]
\newtheorem{proposition}{Proposition}[section]

\theoremstyle{definition}

\newtheorem{remark}{Remark}[section]

\begin{document}
	\title{Gelfand problem on a large spherical cap}
	\author{Yoshitsugu Kabeya}
	\address{Department of Mathematical Sciences, Osaka Prefecture University, Gakuencho, Sakai, 599-8531, Japan}
	\email{kabeya@ms.osakafu-u.ac.jp}
	\author{Vitaly Moroz}
	\address{Department of
		Mathematics, Swansea University, Fabian Way, Swansea SA1~8EN, Wales, UK}
	\email{v.moroz@swansea.ac.uk}
	\date{\today}
	
\keywords{Gelphand problem, spherical cap, torsion function, eigenvalues, spherical harmonics}
\subjclass[2010]{35J60; 33C55, 35R01}

\begin{abstract}
We study the behaviour of the minimal solution to the Gelfand problem on a spherical cap under the Dirichlet boundary conditions. The asymptotic behaviour of the solution is discussed as the cap approaches the whole sphere. The results are based on the sharp estimate of the torsion function of the spherical cap in terms of the principle eigenvalue which we derive in this work.
\end{abstract}

\maketitle

\section{Introduction}

We consider the nonlinear problem
\begin{equation}\label{eq:Gelfand}
\left\{
\begin{aligned}
-\Delta_{\S^N} u&=\lambda f(u)\quad &&\mbox{in $\Omega$},\vspace{10pt}\\
u&=0\quad  &&\mbox{on $\partial\Omega$},
\end{aligned}
\right.
\end{equation}
where $\lambda>0$ is a parameter, $\Delta_{\S^N}$ denotes the Laplace-Beltrami operator on the unit sphere $\S^N\subset\R^{N+1}$ ($N\ge 1$) and $\Omega\subset \S^N$ is a sub-domain in $\S^N$ with a smooth boundary $\partial\Omega\neq\emptyset$. 
The principal Dirichlet eigenvalue of $-\Delta_{\S^N}$ in $\Omega$ is denoted by $\lambda_1(\Omega)>0$, and  $\varphi_{1,\Omega}$ denotes the corresponding positive Dirichlet eigenfunction normalized as $\|\varphi_{1,\Omega}\|_2=1$. By $w_\Omega$ we denote the torsion function of $\Omega$, that is the unique solution of the Dirichlet problem
\begin{equation}\label{eq:torsion}
\left\{
\begin{aligned}
-\Delta w_\Omega&=1\quad \mbox{in}\ \Omega,\vspace{10pt}\\
w_\Omega&=0\quad  \mbox{on}\ \partial\Omega.
\end{aligned}
\right.
\end{equation}
By the standard elliptic regularity, $w_\Omega\in C^2(\Omega)$.

We shall assume that the nonlinearity $f\in C^2(\R)$ is a convex monotone increasing function with $f(0)>0$ that satisfies the assumption
\begin{equation}\label{KO}
\lim_{s\to\infty}\frac{f(s)}{s}=+\infty.
\end{equation}

\noindent
Typical examples include
$$f(s)=\exp(s),\qquad f(s)=(1+s)^p\quad(p>1).$$
Problem \eqref{eq:Gelfand} with this type of nonlinearities is usually referred to as the Gelfand problem. It was 
introduced by Frank-Kamenetskii as a model of thermal explosion in a combustion vessel \cite{FK}, and
became known in the mathematical community due to the chapter written by Barenblatt in a survey by Gelfand \cite[Chapter 15]{Gelfand}.

We denote
\begin{equation}
a_*:=\min_{s>0}\frac{f(s)}{s},\qquad s_*:=\min\Big\{s>0:\frac{f(s)}{s}=a_*\Big\}.
\end{equation}
By convexity and since $f(0)>0$, we observe that $a_*>0$ and $s_*>0$.
Denote 
$$\lambda^*_\Omega:=\sup\big\{\lambda>0: \text{ \eqref{eq:Gelfand} has a classical positive solution}\big\}.$$
The following proposition is standard.

\begin{proposition}\label{p:1}
For each $\lambda\in(0,\lambda^*_\Omega)$, problem \eqref{eq:Gelfand} admits a unique minimal classical positive solution $u_{\lambda}$. Moreover,

\medskip\noindent 
	$i)$ 
	the following estimate holds,
	\begin{equation}\label{e-star}
	\frac{1}{a_*\|w_\Omega\|_\infty}\le\lambda^*_\Omega\le\frac{\lambda_1(\Omega)}{a_*}.
	\end{equation}

	\smallskip\noindent 
	$ii)$ For $\lambda=\lambda^*_\Omega$ problem \eqref{eq:Gelfand} admits a weak extremal solution $u^*>0$ defined as
	\begin{eqnarray}\label{eq:weaks}
	u^*(x):=\lim_{\lambda\to\lambda^*_\Omega} u_{\lambda}(x).
	\end{eqnarray}
	
	\smallskip\noindent 
	$iii)$ For $\lambda>\lambda^*_\Omega$ problem \eqref{eq:Gelfand} admits no weak solutions.
	
\end{proposition}

For the precise definition of the weak solution see \eqref{e-weak}.
In the case when $\Omega$ is a bounded smooth domain in $\R^N$, Proposition \ref{p:1} was essentially proved in Gelfand~\cite{Gelfand}, Keller and Cohen~\cite{Keller67}, Sattinger~\cite{Satt71}, Joseph and Lundgren~\cite{JL72}, Keener and Keller~\cite{Keller74}, Crandall and Rabinowitz~\cite{CR75}. The lower bound in terms of the torsion function in \eqref{e-star} appeared in Bandle~\cite[Theorem 1.1]{Bandle75}, while the upper bound is found in Brezis, Cazenave, Martel and Ramiandrisoa~\cite[Lemma 5]{Brezis96}.
Nonexistence of weak solutions for $\lambda>\lambda^*_\Omega$ is the result in \cite[Corollary 2]{Brezis96}.
We refer to Dupaigne~\cite[Section 3]{Dupaigne} for a detailed exposition and further references.

\begin{remark}\label{r-tor}
	The bound \eqref{e-star} implies implicitly that
	\begin{equation}\label{e-tor-eigen}
	\|w_\Omega\|_\infty>\frac1{\lambda_1(\Omega)}.
	\end{equation}
	It is easy to prove this directly. Indeed, let $\varphi\in C^\infty_0(\Omega)$. Testing \eqref{eq:torsion} against $\frac{\varphi^2}{w_\Omega}$ and using Picone's identity, we conclude that
	\begin{equation}\label{AAP}
	\int_\Omega|\nabla\varphi|^2dS\ge\int_\Omega\frac{\varphi^2}{w_\Omega}dS,
	\end{equation}
	cf.~Agmon \cite[Theorem 3.3]{Agmon} for Riemannian manifolds setting, or Liskevich, Lyakhova and Moroz \cite[Lemma A.9]{LLM}.
	By density, \eqref{AAP} is also valid for all $\varphi\in H^1_0(\Omega)$.
	Then taking $\varphi_{1,\Omega}$ as a test function in \eqref{AAP} and using $\|\varphi_{1,\Omega}\|_2=1$ and the fact that $w_\Omega$ is non-constant, we conclude that
	$$\lambda_1(\Omega)=\int_\Omega|\nabla\varphi_{1,\Omega}|^2dS\ge\int_\Omega\frac{\varphi_{1,\Omega}^2}{w_\Omega}dS>
	\frac{1}{\|w_\Omega\|_\infty}.$$
\end{remark}

\begin{remark}
	As a bi-product of the proof of \eqref{e-star}, we actually establish that the following pointwise estimate holds
	\begin{equation}\label{e-bound}
	\lambda f(0) w_\Omega\le u_\lambda\le s_*\frac{w_\Omega}{\|w_\Omega\|_\infty},
	\end{equation}
	for all $\lambda\le\frac{1}{a_*\|w_\Omega\|_\infty}$. 
\end{remark}

In this work we are primarily interested in the special case when 
\begin{equation}\label{eq:cap}
\Omega_{\eps}:=\Big\{(x_1,x_2,\dots,x_{N+1})\in\R^{N+1}\, \big|\, \sum_{i=1}^{N+1}x_i^2=1,\ \cos((1-\eps)\pi)<x_{N+1}\le 1\Big\}.
\end{equation}
is a geodesic ball, also called a ``spherical cap'', centred at the North Pole of $\S^N$,
and where $\eps\in (0,1)$. Throughout the paper, for $\eps \ll 1$ and $f(\eps), g(\eps) \geq 0$, we use the following asymptotic notation:

\smallskip
$f(\eps)\lesssim g(\eps)$ if there exists $C>0$ independent of $\eps$
such that $f(\eps) \le C g(\eps)$;

\smallskip
$f(\eps)\simeq g(\eps)$ if $f(\eps)\lesssim g(\eps)$ and
$g(\eps)\lesssim f(\eps)$;

\smallskip\noindent
We also use the standard Landau symbols $f = O(g)$ and $f = o(g)$, with the understanding that $f \geq 0$ and $g \geq 0$. As usual, $C,c,c_1$, etc., denote generic positive constants
independent of $\eps$.

It is known by Bandle, Kabeya and Ninomiya~\cite{BKN}, Kabeya, Kawakami, Kosaka and Ninomiya~\cite{3KN}, Kosaka~\cite{Ko} and Macdonald~\cite{McD} (we will review this in Sections~\ref{sec:torsion}),
that the principle eigenvalue of $\Omega_\eps$ satisfy the following asymptotic bound, 
\begin{equation}\label{eq:lambda1}
\lambda_1(\Omega_\eps)=
\left\{
\begin{array}{cl}
c_N\eps^{N-2}(1+o(1))&\quad\mbox{if $N\ge 4$},\vspace{10pt}\\
2\eps(1+o(1))&\quad\mbox{if $N=3$},\vspace{10pt}\\
\frac{1}{2\log\left(\frac{2}{\pi\eps}\right)}(1+o(1))&\quad\mbox{if $N=2$},
\end{array}
\right.
\end{equation}
where coefficients $c_N>0$ do not depend on $\eps>0$ and could be computed explicitly as in Bandle, Kabeya and Ninomiya~ \cite{BKN}.  

In this work we establish the following new estimate on $\|w_{\Omega_\eps}\|_{\infty}$ which shows that the lower bound \eqref{e-tor-eigen} is asymptotically sharp as $\eps\to 0$, and which we believe is of independent interest.

\begin{proposition}\label{prop:torsion} 
Let $N\ge 2$. As $\eps\to 0$, the torsion function $w_{\Omega_\eps}$ of the spherical cap $\Omega_\eps$ satisfies
\begin{equation}\label{eq:torsion0}
\|w_{\Omega_\eps}\|_\infty=\frac{1}{\lambda_1(\Omega_\eps)}+O(1).
\end{equation}
\end{proposition} 

Using \eqref{eq:torsion0}, from \eqref{e-star} we deduce a sharp
asymptotic estimate on the critical parameter $\lambda^*_{\Omega_\eps}$ in the Gelfand problem \eqref{eq:Gelfand}.

\begin{theorem} 
Let $N\ge 2$. As $\eps\to 0$,
\begin{equation}\label{eq:torsion-bound}
\lambda^*_{\Omega_\eps}=\frac{\lambda_1(\Omega_\eps)}{a_*}(1+o(1)).
\end{equation}
\end{theorem}

The paper is organised as follows. In Section 2 we sketch the proof of Proposition \ref{p:1}.
In Section 3 we review several fundamental properties of the special functions
and outline the derivation of the torsion function estimates which lead to the proof of Proposition \ref{prop:torsion}. In Section 4 we prove a technical Lemma \ref{prop:1}.  Finally, in Sections 5, 6 
we discuss several technical properties of the special functions which were used in the earlier proofs.
\medskip

\section{Proof of Proposition \ref{p:1}}

The proof of Proposition \ref{p:1} is standard, cf.~ Dupaigne~\cite[Propositions 3.3.1, 3.3.2]{Dupaigne} for a detailed exposition in the case of bounded smooth subdomains of $\R^N$. We present a sketch of the proof for completeness.

\begin{lemma} \label{l:int1}
	Assume \eqref{eq:Gelfand} admits positive,
	classical super-solution $\bar u>0$. Then \eqref{eq:Gelfand} admits a unique minimal,  
	positive classical solution $u_{\lambda}\in C^2 (\Omega)$, such that
	\begin{equation}\label{e-bound-l}
	\lambda f(0) w_\Omega\le u_\lambda\le \bar u.
	\end{equation}
\end{lemma}

\begin{proof}
	The minimal solution $u_{\lambda}$ of \eqref{eq:Gelfand} could be obtained by a monotone iteration arguments.
	Namely, for $n\in\N$ consider a sequence of functions $\{u_n\}_{n=0}^{\infty}$ with 
	$u_0=0$ and $u_n$ defined as
	\begin{eqnarray}\label{eq:app2c}
	\left\{
	\begin{array}{rclcc}
	-\Delta_{\S^N} u_n &= &\lambda f( u_{n-1})& \mbox{in} & \Omega,\\
	u_{n} &= &0  & \mbox{on} & \partial \Omega.
	\end{array}
	\right.
	\end{eqnarray}
	For each $n$ problem \eqref{eq:app2c} 
	admits a unique solution $u_n\in C^2(\bar\Omega)$ and $u_{n}\ge u_{n-1}$ in $\Omega$ for every $n\in\N$. In particular,
	$$u_n\ge u_1=\lambda f(0)w_\Omega.$$
	We now define the minimal solution of \eqref{eq:Gelfand} as
	\begin{eqnarray}\label{eq:app2d}
	u_{\lambda}(x):=\lim_{n\to\infty} u_n(x).
	\end{eqnarray}
	By the arguments similar to Sattinger~\cite[Theorem 2.1]{Satt71} we conclude that  $u_{\lambda}$ defined by \eqref{eq:app2d} satisfies  $\bar u\ge u_{\lambda}>0$
	in $\Omega$, 	belongs to $C^2(\bar \Omega)$ and solves \eqref{eq:Gelfand} classically.
\end{proof}

The next lemma uses the notion of a weak solution. Following Brezis, Cazenave, Martel and Ramiandrisoa~\cite{Brezis96}, we define a weak solution of \eqref{eq:Gelfand} as  a non-negative function $u\in L^1(\Omega)$ such
that $f(u) \delta_{\partial\Omega} \in L^1(\Omega),$ where $\delta_{\partial\Omega}(x):={\rm dist}(x,\partial\Omega)$ is the distance from $x$ to the boundary of $\Omega$ and
\begin{eqnarray}\label{e-weak}
-\int_{\Omega} u \Delta_{\S^N}\varphi dS=\lambda\int_{\Omega}f(u) \varphi dS,
\end{eqnarray}
for all $\varphi\in C^2(\bar\Omega)$ with $\varphi=0$ on $\partial\Omega$.

\begin{lemma} \label{l:int2}
	Problem \eqref{eq:Gelfand} admits a minimal classical solution $u_{\lambda}$
	for $0<\lambda<\lambda^*_\Omega<\infty$. Moreover, the extremal solution
	$$
	u^*(x):=\lim_{\lambda\to\lambda^*_\Omega} u_{\lambda}(x),
	$$
	is a weak solution of \eqref{eq:Gelfand}.
\end{lemma}
\begin{proof}
	
	First observe that $u_{\lambda}$ is a non-decreasing  function of $\lambda.$ 
	This follows from the fact that
	$u_{\lambda^{\prime}}$ is a super-solution for  problem \eqref{eq:Gelfand} with
	$\lambda<\lambda^{\prime}.$ Hence, if \eqref{eq:Gelfand} with $\lambda=\lambda^{\prime}$
	admits a classical  solution, then \eqref{eq:Gelfand} admits a classical solution for $\lambda\in(0,\lambda^{\prime}]$.
	
	Next, it is easy to see that a multiple of the torsion function $w_\Omega$ is a super-solution for \eqref{eq:Gelfand} for small enough $\lambda$.
	Let 
	$$s_*:=\min\Big\{s>0:\frac{f(s)}{s}=a_*\Big\}>0.$$
	Set
	$$t^*:=\frac{s_*}{\|w_\Omega\|_\infty}.$$
	and take $\overline{u}:=t^*w_\Omega$ and $\lambda^\prime:=\frac{1}{a_*\|w_\Omega\|_\infty}$.
	Then
	\begin{equation}
	-\Delta_{\S^N}\overline{u}=t^*=t^*\frac{f(s_*)}{a_*s_*}=\frac{f(t^*\|w_\Omega\|_\infty)}{a_*\|w_\Omega\|_\infty}\ge \lambda^\prime f(\overline{u}),
	\end{equation}
	i.e. $\overline{u}$ is a positive supersolution for \eqref{eq:Gelfand} for all $\lambda\le\lambda^\prime$. In particular, this establishes the lower bound in \eqref{e-star}.

	Now let us show that $\lambda^*_\Omega<\infty$. Similarly to Brezis, Cazenave, Martel and Ramiandrisoa~\cite[Lemma 5]{Brezis96}, we test \eqref{eq:Gelfand} against the $\varphi_{1,\Omega}$. Then, 
	\begin{equation}
	\lambda_1(\Omega)\int_\Omega u_\lambda\varphi_{1,\Omega}\,dS=-\int_{\Omega} u_\lambda \Delta_{\S^N}\varphi_{1,\Omega}\,dS=\lambda\int_\Omega f(u_\lambda)\varphi_{1,\Omega}\, dS\ge\lambda a_*\int_\Omega u_\lambda\varphi_{1,\Omega}\,dS.
	\end{equation}
	We conclude that $\lambda_1(\Omega)\ge\lambda a_*>0$, which establishes the upper bound in \eqref{e-star}.
	
	Finally, proceeding as in the proof of \cite[Lemma 5]{Brezis96} or Dupaigne~\cite[Proposition 3.3.2]{Dupaigne} and using \eqref{KO}, we recover that $u^*$ is 	a weak solution of \eqref{eq:Gelfand}.
\end{proof}

\begin{lemma} \label{l:int3}
	Problem \eqref{eq:Gelfand} admits no weak solutions for $\lambda>\lambda^*_\Omega$.
\end{lemma}

\begin{proof}
Similar to  Brezis, Cazenave, Martel and Ramiandrisoa~\cite[Theorem 3]{Brezis96} or Dupaigne~\cite[Proposition 3.3.1]{Dupaigne} -- we omit the details.
\end{proof}

\section{Torsion function estimates}\label{sec:torsion}

In the rest of the paper we always work on the spherical cap $\Omega_\eps$. When there is no ambiguity, 
we suppress the dependence on $\eps$ and write $\lambda_1$, $\varphi_1$ instead of $\lambda_1(\Omega_\eps)$, $\varphi_{1,\Omega_\eps}$. We always assume that $\eps\in(0,1)$ and is close to $0$.

\subsection{Lower dimensions}
Consider the stereographic projection from $\S^N$ onto $\R^N$ that maps South Pole to infinity. 
Then the torsion function equation \eqref{eq:torsion} on the spherical cap $\Omega_\eps$ is transformed into linear ODE
\begin{equation}\label{eq:rad}
\left\{
\begin{aligned}
-\big(r^{N-1}p^{N-2}(r)w^\prime\big)^\prime&=r^{N-1}p^N(r),\qquad 0<r<R,\vspace{10pt}\\
w'(0)=0,\quad\quad w(R)&=0.
\end{aligned}
\right.
\end{equation}
where $r=\tan(\theta/2)$ and $\theta\in(0,\pi)$ is the geodesic distance to the North Pole, 
$$p(r):=\frac{2}{1+r^2},$$
and the spherical cap $\Omega_\eps$ on $\S^N$ is mapped onto the ball $B_{R}=\{x\in \R^N\,|\,|x|=R\}$
where 
\begin{equation}\label{eq:Reps}
R=\tan\Big(\frac{(1-\eps)\pi}{2}\Big)=\frac{2}{\pi\eps}+o(\eps)\quad\text{as $\eps\to 0$},
\end{equation}
see e.g.~Bandle and Benguria~\cite{BB} for a discussion. In lower dimensions solutions of \eqref{eq:rad} can be found explicitly.  
Taking into account \eqref{eq:Reps} and \eqref{eq:lambda1}, for $N=2$ we obtain
\begin{align}
w_{\Omega_\eps}(r)&=\log\frac{1+R^2}{1+r^2},\\ 
w_{\Omega_\eps}(0)&=2\log\big(\frac{2}{\pi\eps}\big)+o(\eps)=\frac1{\lambda_1(\Omega_\eps)}+O(1)\quad\text{as $\eps\to 0$},\label{eN2}
\end{align}
and for $N=3$ we obtain
\begin{align}
w_{\Omega_\eps}(r)&=\frac{1-r^2}{2r}\tan^{-1}(r)-\frac{1-R^2}{2R}\tan^{-1}(R),\\
w_{\Omega_\eps}(0)&=
\frac1{2\eps}+o(1)=\frac1{\lambda_1(\Omega_\eps)}+O(1)\quad\text{as $\eps\to 0$}.\label{eN3}
\end{align}
This settles Proposition \ref{prop:torsion} for $N=2,3$.
In higher dimensions no explicit closed form expression for the torsion function on the spherical cap seems to be available. Instead, we are going to use the Fourier method to represent the torsion function.

\subsection{General Setting}

Let
$$0<\lambda_1<\lambda_2\le\lambda_3\le \dots\le \lambda_\ell\le \dots\qquad (\ell\in\N)$$
be the Dirichlet eigenvalues of $-\Delta_{\S^N}$ on the spherical cap $\Omega_\eps$, and 
$\varphi_\ell$ be the corresponding eigenfunctions normalized as 
$\|\varphi_\ell\|_2=1$. That is, $\varphi_\ell$ is a regular solution to the Dirichlet problem 
\begin{equation}\label{eq:eigenf}
-\Delta \varphi_\ell=\lambda_\ell\varphi_\ell,\qquad \varphi_\ell\in H^1_0(\Omega_\eps).
\end{equation}
It is well-known that the sequence $\{\varphi_\ell\}$ forms the complete ortho-normal system in $L^2(\Omega)$. 
Then we see that the torsion function $w_{\Omega_\eps}$ defined as the unique solution to \eqref{eq:torsion} is expressed as
\begin{equation}\label{sol}
w_{\Omega_\eps}=\sum_{\ell=1}^{\infty}\frac{(1,\varphi_\ell)}{\lambda_\ell}\varphi_\ell, 
\end{equation}
where 
$$
(\Phi,\Psi)=\int_{\Omega_\eps}\Phi\Psi\, dS.
$$
Indeed, due to the Fourier expansion of $1$ and the definition of the eigenfunctions, we have
$$
-\Delta w_{\Omega_\eps}=\sum_{\ell=1}^{\infty}\frac{(1,\varphi_\ell)}{\lambda_\ell}(-\Delta \varphi_\ell)= \sum_{\ell=1}^{\infty}\frac{(1,\varphi_\ell)}{\lambda_\ell}(\lambda_\ell \varphi_\ell)=\sum_{\ell=1}^{\infty}(1,\varphi_\ell)\varphi_\ell=1.
$$
We are going to estimate the right-hand side of \eqref{sol}. 
For clarity, we first outline the arguments in the case $N=3$, which is already covered by \eqref{eN3}.
 
\subsection{Three dimensional case, ``radial'' eigenfunctions.}\label{s32}
For $N=3$ we have 
\begin{multline*}
\Omega_\eps=
\big\{(\sin\theta\sin\phi\cos\psi,\sin\theta\sin\phi\sin\psi, \sin\theta\cos\phi, \cos\theta)\,|\\  0\le \theta<(1-\eps)\pi,  0\le\phi\le \pi, 0\le \psi<2\pi\big\}.
\end{multline*}
First, we consider the solutions to \eqref{eq:torsion} depending only on $\theta$.  In this case, the Laplacian yields to 
$$
\Delta_{{\mathbb S}^3} \varphi=\frac{1}{\sin^2\theta}(\varphi' \sin^2\theta)'.
$$
Hence by the direct calculation, we see that the eigenfunctions depending 
only on $\theta$ are  of the form 
\begin{equation}\label{eq:ef}
\varphi_j(\theta)=K_j\frac{\displaystyle{\sin \frac{j}{1-\eps}\theta}}{\sin\theta}
\end{equation}
where $K_j$ are normalization constants and the corresponding eigenvalues $\lambda_j$ are 
\begin{equation}\label{eq:lj}
\lambda_j=\frac{j^2}{(1-\eps)^2}-1=\frac{j^2-1+2\eps-\eps^2}{(1-\eps)^2}=2\eps j^2+O(\eps^2).
\end{equation} 
Note that $\lim_{\eps\to 0}\lambda_j=0$ only when $j>1$, and that
\begin{equation}\label{eq:phi-infty}
\|\varphi_j\|_\infty=\varphi_j(0)=K_j\frac{j}{1-\eps}+o(\eps).
\end{equation}

The inner product in our case is given by
$$
(\Phi,\Psi)=\int_{\Omega}\Phi\Psi\, dS=4\pi\int_0^{(1-\eps)\pi}\Phi(\theta)\Psi(\theta)\sin^2\theta \, d\theta.
$$ 
Then we have
$$
\begin{aligned}
(1,\varphi_j)&=4\pi K_j\int_0^{(1-\eps)\pi}\sin \frac{j}{1-\eps}\theta \sin \theta \, d\theta\vspace{10pt}\\
&=2\pi K_j\int_0^{(1-\eps)\pi}\left(\cos \frac{j-1+\eps}{1-\eps}\theta-\cos \frac{j+1-\eps}{1-\eps}\theta\right)\, d\theta.
\end{aligned}
$$
Hence we obtain
$$
\begin{aligned}
(1,\varphi_j)&=\displaystyle{(-1)^j2\pi K_j(1-\eps)\left\{\frac{\sin (j-2)\eps\pi+O(\eps^3)}{j-1+\eps}-\frac{\sin j\eps\pi+O(\eps^3)}{j+1-\eps}\right\}}\vspace{10pt}\\
&=\displaystyle{-(-1)^{j}2\pi K_j(1-\eps)\frac{2\eps\pi+O(\eps^3)}{(j-1+\eps)(j+1-\eps)}.}
\end{aligned}
$$ 
As for $K_j$, we have
$$
(\varphi_j,\varphi_j)=4\pi K_j^2\int_0^{(1-\eps)\pi}\sin^2\frac{j}{1-\eps}\theta\, d\theta=2\pi K_j^2\int_0^{(1-\eps)\pi}\left(1-2\cos\frac{2j}{1-\eps}\theta\right)\, d\theta
$$
and thus, from
$$
1= (\varphi_j,\varphi_j)=2\pi K_j^2 (1-\eps)\pi,
$$
we obtain
\begin{equation}\label{eq:cj}
K_j=\frac{1}{\sqrt{2}\pi}+O(\eps).
\end{equation}
We hence conclude that 
\begin{equation}\label{eq:1j}
(1,\varphi_j)=\left\{
\begin{aligned}
\sqrt{2}+O(\eps),\quad j=1,\vspace{10pt}\\
\frac{(-1)^{j+1}2\sqrt{2}}{(j+1)(j-1)}\eps+O(\eps^2), \quad j\ge 2.
\end{aligned}
\right.
\end{equation}
From \eqref{sol} this implies that 
\begin{equation}\label{eq:1j2}
w_{\Omega_\eps}=\frac{\sqrt{2}\pi}{\lambda_1}\varphi_1+\sum_{j=2}^{\infty}\frac{(1,\varphi_j)}{\lambda_j}\varphi_j,
\end{equation}
and taking into account \eqref{eq:1j}, \eqref{eq:lj}, \eqref{eq:phi-infty} and \eqref{eq:cj} we conclude that
\begin{equation}\label{eq:1j3}
\left|\sum_{j=2}^{\infty}\frac{(1,\varphi_j)}{\lambda_j}\varphi_j\right|\le \sum_{j=2}^{\infty}O\big(j^{-4}\big)\|\varphi_j\|_\infty\le C.
\end{equation}
Since $\|\varphi_1\|_\infty=K_1(1-\eps)^{-1} $, using \eqref{eq:cj} we estimate
$$
\|w_{\Omega_\eps}\|_{\infty}\le \frac{\sqrt{2}\pi}{\lambda_1}\|\varphi_1\|_\infty+O(1)=\frac{1}{\lambda_1}+O(1).
$$
Combined with \eqref{e-tor-eigen} and \eqref{eq:lj} this leads to the two sided bound
\begin{equation}\label{e-olittle}
\|w_{\Omega_\eps}\|_{\infty}=\frac1{\lambda_1}+O(1)=\frac{1}{2\eps}+O(1).
\end{equation}

\begin{remark}
	We will see in Lemma~\ref{le:conv} that  
	$$\lim_{\eps\to 0}\|\varphi_1\|_\infty=\lim_{\eps\to 0}\varphi_1(0)$$
	holds in all dimensions $N\ge 2$.
	Thus, we obtain 
	$$\lim_{\eps\to 0}(1,\varphi_1)\|\varphi_1\|_\infty=1$$ and then
	\begin{equation}\label{e-olittle2}
	\|w_{\Omega_\eps}\|_{\infty}=\frac1{\lambda_1}+O(1),
	\end{equation}
	as in the cases $N=2,3$ above.
\end{remark}

\subsection{General dimensions, ``radial'' eigenfunctions.}\label{sec:gene}

In the general case $N\ge 3$ 
the spherical cap we consider is represented as follows:
\begin{equation*}
\Omega_{\eps}=\big\{(x_1,x_2,\dots,x_{N+1})\, |\, \sum_{i=1}^{N+1}x_i^2=1,\ \cos(1-\eps)\pi<x_{N+1}\le 1\big\}.
\end{equation*}
We introduce the polar coordinates 
$$
\left\{
\begin{array}{l}
x_1=\displaystyle{(\prod_{j=1}^{N-1}\sin\theta_j)\sin\phi},\vspace{10pt}\\
x_2=\displaystyle{(\prod_{j=1}^{N-1}\sin\theta_j)\cos\phi}, \vspace{10pt}\\
x_{N-k+1}=\displaystyle{(\prod_{j=1}^{k}\sin\theta_j)\cos\theta_{k+1}},\quad k=1,2,\dots, N-2,\vspace{10pt}\\
x_{N+1}=\cos\theta_1.
\end{array}
\right.
$$
Then the Laplace-Beltrami operator is expressed as 
$$
\begin{aligned}
\Delta_{\S^N}= &\frac{1}{\sin^{N-1}\theta_1}\frac{\partial}{\partial \theta_1}\left[\sin^{N-1}\theta_1\frac{\partial}{\partial \theta_1}\right]\\
&+\frac{1}{\sin^2\theta_1\sin^{N-2}\theta_2}\frac{\partial}{\partial \theta_2}\left[\sin^{N-2}\theta_2\frac{\partial}{\partial \theta_2}\right]\\
&+\cdots\\
&+ \frac{1}{\sin^2\theta_1\sin^2\theta_2\cdots \sin^2\theta_{N-2}\sin\theta_{N-1}}\frac{\partial}{\partial \theta_{N-1}}\left[\sin\theta_{N-1}\frac{\partial}{\partial \theta_{N-1}}\right]\\
&+ \frac{1}{\sin^2\theta_1\cdots \sin^2\theta_{N-1}}\frac{\partial^2}{\partial \phi^2}.
\end{aligned}
$$
Hence
$$
\Delta_{\S^N}= \frac{1}{\sin^{N-1}\theta_1}\frac{\partial}{\partial \theta_1}\left[\sin^{N-1}\theta_1\frac{\partial}{\partial \theta_1}\right]+\frac{1}{\sin^2\theta_1}\Delta_{S_{N-1}}, \quad \Delta_{\S^1}=\frac{\partial^2}{\partial \phi^2}.
$$
Here we note that $\Delta_{\S^{N-1}}$ is a Laplace-Beltrami operator on $\mathbb{S}^{N-1}$ with the coordinates $\{\theta_2,\cdots,\theta_{N-1},\phi\}$.
\par
It is known (cf. Bandle, Kabeya and Ninomiya~\cite{BKN}) that
the Dirichlet eigenfunctions for $-\Delta_{\S^N}$ on $\Omega_\eps$ depending only on $\theta\in [0,(1-\eps)\pi]$, 
which are regular solutions to 
\begin{equation}\label{eq:rad-form}
\frac{d}{d\theta}(\sin^{N-1} \theta \frac{d}{d\theta}\varphi)+\lambda \sin^{N-1}\theta \varphi=0
\end{equation} 
are expressed as a constant multiple of 
\begin{equation}\label{eq:legendre}
\varphi_j(\theta_1)=K_{j, N}\frac{P_{\nu_j}^{\mu}(\cos\theta_1)}{(\sin \theta_1)^{(N-2)/2}},\qquad j=1,2,3,\dots
\end{equation}
with 
$$\mu=(-1)^{N}\tfrac{N-2}{2}
$$
and the normalizing constant $K_{j, N}$ so that $\|\varphi_j\|_2=1$. Here,
$P_\nu^{\mu}$ is the associated Legendre function defined in terms of the Gauss hypergeometric function $F(\alpha,\beta,\gamma;x)$ as follows:
\begin{equation}\label{eq:P-F}
P_\nu^{\mu}(t)=\frac{1}{\Gamma(1-\mu)}\left(\frac{1+t}{1-t}\right)^{\mu/2}F\left(-\nu,\nu+1,1-\mu;\frac{1-t}{2}\right), \quad t\in (-1,1),
\end{equation}
when $\mu\not\in {\mathbb N}$ with $F(\alpha,\beta,\gamma;x)$ being defined as
$$
 F(\alpha,\beta,\gamma;x)=\frac{\Gamma(\gamma)}{\Gamma(\alpha)\Gamma(\beta)}\sum_{n=0}^{\infty}\frac{\Gamma(\alpha+n)\Gamma(\beta+n)}{\Gamma(\gamma+n)}\frac{x^n}{n!}.
 $$
 If $\mu$ is a natural number, $P_\nu^{\mu}(t)$ is defined as 
\begin{equation}\label{eq:P-Fnatu}
P_\nu^{\mu}(t)=\frac{\Gamma(\mu+\nu+1)}{\mu!2^{\mu}\Gamma(1+\nu-\mu)}(1-t^2)^{\mu/2}F\left(\mu-\nu,\mu+\nu+1,\mu+1;\frac{1-t}{2}\right), \quad t\in (-1,1).
\end{equation}
Also, it is known that $P_\nu^{\mu}(t)$ is a solution to the associated Legendre differential equation
 \begin{equation}\label{eq:ass-legendre}
 (1-t^2)\frac{d^2u}{dt^2}-2t\frac{du}{dt}+\left\{\nu(\nu+1)-\frac{\mu^2}{1-t^2}\right\}u=0.
 \end{equation}
 Also,  $P_\nu^{\mu}(t)/(1-t^2)^{-\mu/2}$ is a solution to the ultra-sphere differential equation
  \begin{equation}\label{eq:ultra-sphere}
 (1-t^2)\frac{d^2u}{dt^2}-2(\mu+1)t\frac{du}{dt}+(\nu-\mu)(\nu+\mu+1)u=0.
 \end{equation}
\par
If we consider all the eigenfunctions, which depend on more than two variables (we discuss this in details later in Sections \ref{s35}, \ref{s36}),  those of $-\Delta$ in $\Omega_{\eps}$ are represented in terms of the polar coordinate as
\begin{equation}\label{eq:n-eigenf_n}
\Tilde \Phi_\ell=K_\ell\frac{\hat{P}_\nu^{k_2+(N-2)/2}(\cos\theta_1)}{(\sin \theta_1)^{(N-2)/2}}\prod_{j=1}^{N-2}\frac{\hat P_{k_{{}_{N-j}}+(j-1)/2}^{k_{{}_{N-j+1}}+(j-1)/2}(\cos\theta_{N-j})}{(\sin \theta_{N-j})^{(j-1)/2}}\times\left\{
\begin{array}{l}
\cos k_{N}\phi,\vspace{10pt}\\
\sin k_{N}\phi,
\end{array}
\right.
\end{equation}
where  $\ell$ is a nonnegative integer and where $\{k_s\}$ is a sequence of integers such that 
\begin{equation}\label{k-order}
0\leq k_N\leq k_{N-1}\leq\dots\leq k_{j+1}\leq k_j\leq\dots\leq k_2\leq \ell.
\end{equation}
The real number $\nu$ must be chosen so that $\hat{P}_\nu^{k_2+(N-2)/2}$ solves \eqref{eq:eigenf}. Here we define 
$$
{\hat P}_{\nu}^{\mu}(\cos \theta)=\left\{
\begin{aligned}
P_{\nu}^{\mu}(\cos\theta_1)&\ \mbox{if}\ \mu\ \mbox{is\ an\ integer,}\vspace{10pt}\\
P_{\nu}^{-\mu}(\cos\theta_1)&\ \mbox{if}\ \mu\ \mbox{is\ a\ half\ integer,}\
\end{aligned}
\right.
$$
where $\theta_1$ varies in $[0, (1-\eps)\pi]$, $\theta_i\in [0,\pi]$ ($i=2,3,\dots,N-1)$ and $\phi\in[0,2\pi]$. 

Note that $\ell=0$ corresponds to the first eigenvalue and that the corresponding eigenfunction depends only on $\theta_1$ according to \eqref{k-order}. First, we study the case when all the eigenfunctions depend only on $\theta_1$. 
Then, for the Dirichlet problem,  by Bandle, Kabeya and Ninomiya~\cite{BKN} or Macdonald~\cite{McD}, there holds 
\begin{equation}\label{eq:nu_ep-n}
\nu_j=(j-1)+\frac{N-2}{2}+C_j\eps^{N-2}+o(\eps^{N-2}),
\end{equation}
where $C_j$ is a constant explicitly determined in \cite{BKN}, which has a complicated expression. 

The relation between $\nu_j$ and the eigenvalue $\lambda_j$ is expanded as 
\begin{equation}\label{eq:lambda_n}
\lambda_j=\nu_j(\nu_j+1)-\frac{N(N-2)}{4}=(j-1)(j+N-2)+(2j+N-3)C_j\eps^{N-2}+o(\eps^{N-2}).
\end{equation}
Note that the first eigenvalue 
\begin{equation}\label{eq:1steigen}
\lambda_1=(N-1)C_1\eps^{N-2}+o(\eps^{N-2})
\end{equation}
is close to zero.

In order to estimate the maximum of the torsion function $\|w_{\Omega_\eps}\|_{\infty}$, we need the following formula on the definite integral formula, which is valid for any dimension, and which is written on p.~129 in  Moriguchi, Udagawa and Hitotsumatsu~\cite{MUH} as
\begin{equation}\label{eq:definite}
\begin{aligned}
&\int_0^{\pi}P_\nu^{\mu}(\cos\theta)\sin^{\alpha-1}\theta\, d\theta\vspace{10pt}\\
&\hspace{20pt}=\frac{2^{\mu}\pi \Gamma((\alpha+\mu)/2)\Gamma((\alpha-\mu)/2)}{\Gamma((\alpha-\nu)/2)\Gamma((\nu+\alpha+1)/2)\Gamma((-\mu-\nu+1)/2)\Gamma((\nu-\mu+2)/2)}.
\end{aligned}
\end{equation} 
The expression is valid for any $\mu,\ \alpha\in{\mathbb R}$ such that $|\mu|<\alpha$. In this case, $|\mu|=(N-2)/2<N/2+1$. Thus, \eqref{eq:definite} is applicable.

\subsection{All the eigenfunctions. Three dimensional case.}\label{s35}
If we consider the eigenfunctions which depends on $\theta_i$ and/or $\phi$, first we consider the case $N=3$. Then according to \eqref{eq:n-eigenf_n}, we have
\begin{equation}\label{eq:3eigen}
\Phi=K_j\frac{\hat{P}_\nu^{m+1/2}(\cos\theta_1)}{(\sin \theta_1)^{1/2}}P_{m}^{k}(\cos\theta_2)(c_1\cos k\phi+c_2\sin k\phi).
\end{equation}
In the three dimensional case, the sphere element $dS$ is expressed as 
$$
dS=\sin^{2}\theta_1\sin\theta_2\, d\theta_1d\theta_2d\phi.
$$
If $k$ is a natural number, it is very easy to see 
that
$$
\begin{aligned}
&\int_{\Omega_\eps}\hat\Phi\, dS\vspace{10pt}\\
&=\int_0^{(1-\eps)\pi}\!\!\!\int_0^{\pi}\!\!\!\int_0^{2\pi}\frac{\hat{P}_\nu^{m+1/2}(\cos\theta_1)}{(\sin \theta_1)^{1/2}}P_{m}^{k}(\cos\theta_2)(c_1\cos k\phi+c_2\sin k\phi)\, \sin^{2}\theta_1\sin\theta_2 d\phi d\theta_2d\theta_1\vspace{10pt}\\
&=0
\end{aligned}
$$
since $\phi$ varies from $0$ to $2\pi$. If $k=0$ in \eqref{eq:3eigen},  then we consider $P_m^{0}(\cos\theta_2)$. In this case, we need to calculate
$$
\int_0^{\pi}P_m^{0}(\cos\theta_2)\sin \theta_2\, d\theta_2
$$
with some natural number $m$. 
In \eqref{eq:definite}, we take $\mu=0,\ \nu=m,\ \alpha=2$. Hence, we have
$$
\frac{\alpha-\nu}{2}=1-\frac{m}{2},\quad \frac{-\mu-\nu+1}{2}=-\frac{m-1}{2}.
$$
If $n$ is even, the former is nonpositive integer and if $n$ is odd, the latter is so. Hence, due to \eqref{eq:definite}, we obtain
$$
\int_0^{\pi}P_n^{0}(\cos\theta_2)\sin \theta_2\, d\theta_2=0.
$$
Hence when $N=3$, it is sufficient to consider the ``radial" case only. 

\subsection{All the eigenfunctions. Four dimensional case and the general case.}\label{s36}
Similarly to the three dimensional case, the eigenfunction is of the form
\begin{equation}\label{eq:4eigen}
\hat \Phi=K_j \frac{\hat{P}_{\nu}^{n+1}(\cos\theta_1)}{\sin \theta_1}\frac{\hat{P}_{n+1/2}^{m+1/2}(\cos\theta_2)}{(\sin \theta_2)^{1/2}}P_{m}^{k}(\cos\theta_3)(c_1\cos k\phi+c_2\sin k\phi).
\end{equation}
In the four dimensional case, the sphere element $dS$ is expressed as 
$$
dS=\sin^{3}\theta_1\sin^{2}\theta_2\sin\theta_3\, d\theta_1d\theta_2d\phi.
$$
As in the three dimensional case, if one of $k$, $m$ is positive, then we have 
$$
\begin{aligned}
&\int_{\Omega_\eps}\hat\Phi\, dS&\vspace{10pt}\\
&=\int_0^{(1-\eps)\pi}\!\!\!\int_0^{\pi}\!\!\!\int_0^{\pi}\!\!\!\int_0^{2\pi}\frac{\hat{P}_{\nu}^{n+1}(\cos\theta_1)}{\sin \theta_1}\frac{\hat{P}_{n+1/2}^{m+1/2}(\cos\theta_2)}{(\sin \theta_2)^{1/2}}P_{m}^{k}(\cos\theta_3)\left\{
\begin{aligned}
&c_1\cos k\phi\vspace{10pt}\\
&c_2\sin k\phi
\end{aligned}
\right\}\times\vspace{10pt}\\
&\hspace{200pt} \sin^{3}\theta_1\sin^{2}\theta_2\sin\theta_3 d\phi d\theta_3 d\theta_2d\theta_1\vspace{10pt}\\
&=0.
\end{aligned}
$$
Thus, we consider $k=m=0$ and calculate 
$$
\int_0^{\pi}\hat{P}_{n+1/2}^{1/2}(\cos\theta_2)\sin^{3/2}\theta_2\, d\theta_2.
$$
In this case, we note that $\hat{P}_{n+1/2}^{1/2}(\cos\theta_2)=P_{n+1/2}^{-1/2}(\cos\theta_2)$. Then we apply \eqref{eq:definite} with 
$$\
\alpha=\frac{5}{2}, \quad \mu=-\frac{1}{2},\quad \nu=n+\frac{1}{2}.
$$ 
Hence we have 
$$
\frac{\alpha-\nu}{2}=-\frac{n-2}{2}, \quad -\frac{\mu+\nu-1}{2}=-\frac{n-1}{2}.
$$
If $n$ is an even integer, then $(\alpha-\nu)/2$ is a nonpositive integer and if $n$ is odd, so is  $-(\mu+\nu-1)/2$. 
Thus, in the four dimensional case also, we see that
$$
\int_0^{\pi}\hat{P}_{n+1/2}^{1/2}(\cos\theta_2)\sin^{3/2}\theta_2\, d\theta_2=0
$$
if $n\ge 1$ and we have only to consider the eigenfunctions depending only on $\theta_1$.

Inductively, we proved that we have only to consider the ``radial" eigenfunctions to estimate the torsion function. Thus in \eqref{eq:n-eigenf_n}, we have only to consider the case 
$$
k_2=k_3=\dots=k_N=0,
$$
 that is, we consider the eigenfunctions of the form 
$$
\frac{P_\nu^{(N-2)/2}(\cos\theta_1)}{(\sin\theta_1)^{(N-2)/2}}
$$ 
as in \eqref{eq:legendre}. 

\subsection{Proof of Proposition~\ref{prop:torsion} completed.}
To complete the proof of Proposition~\ref{prop:torsion} in general dimensions $N\ge 4$, we use the representation of the torsion function $w_{\Omega_\eps}$ as in \eqref{sol}.  
To calculate $(1, \varphi_j)$, we note that 
$$
(1,\varphi_j)=\omega_{N-1}\int_0^{(1-\eps)\pi}\varphi_j(\theta)\sin^{N-1}\theta\, d\theta
$$
where $\omega_{N-1}$ stands for the area of the unit sphere ${\mathbb S}^{N-1}.$ Hence, 
$$
(1,\varphi_j)=\omega_{N-1}K_j\int_0^{(1-\eps)\pi}P_{\nu}^{\mu}(\cos\theta)(\sin \theta)^{N/2}\, d\theta.
$$
We recall \eqref{eq:definite} here and the validity of the formula. The conditions are $\mu,\ \alpha\in{\mathbb R}$ such that $|\mu|<\alpha$. In this case, $|\mu|=(N-2)/2<N/2+1$. Thus, \eqref{eq:definite} is applicable.
When $j\ge 2$, we should be careful. Since $\Gamma(-n)=\infty$ $(n=0,1,2,\dots)$ (see (i) of Lemma~\ref{le:Gamma} below), or more precisely, $z=-n$ is the pole of order one for $\Gamma(z)$, if $(-\mu-\nu+1)/2$ or $(\alpha-\nu)/2$ are non-positive integers, 
we regard 
$$
\int_0^{\pi}P_\nu^{\mu}(\cos\theta)\sin^{\alpha-1}\theta\, d\theta=0.
$$
In our case,  $(-\mu-\nu+1)/2$ or $(\alpha-\nu)/2$ are not exactly non-positive integers, but are close to those.  This means that the value of the desired definite integral is small enough.

Indeed, when $N=2m$ $(m=2,3,\dots$), we have 
$$
\mu=m-1, \quad \nu \simeq j+m-2+C_j\eps^{N-2}, \quad \alpha=m+1,
$$
$$
\nu+\alpha+1\simeq 2m+j+C_j\eps^{N-2},\quad \nu-\mu+2\simeq j+1+C_j\eps^{N-2}.
$$
Then 
$$
\alpha-\nu \simeq 3-j-C_j\eps^{N-2},\quad -\mu-\nu+1\simeq -j-2m+4-C_j\eps^{N-2}.
$$
For $j=1$, these do not converge to a nonpositive integer.  
Hence we have 
$$
\begin{aligned}
\int_0^{\pi}P_\nu^{\mu}(\cos\theta)\sin^{m}\theta\, d\theta&=\frac{2^{N/2-1}\pi\Gamma(\frac{N}{2})\Gamma(1)}{(\Gamma(1))^2\Gamma ((\frac{N+1}{2})\Gamma(\frac{3-N}{2})}+O(\eps^{N-2})\vspace{10pt}\\
&=\frac{(-1)^{N/2+1}2^{N/2}(\frac{N}{2}-1)!}{N-1}+O(\eps^{N-2})\ne0
\end{aligned}
$$
and
$$
\begin{aligned}
&\int_0^{\pi}P_\nu^{\mu}(\cos\theta)\sin^{m}\theta\, d\theta \simeq\vspace{10pt}\\
&\hspace{40pt} \frac{2^{m-1}\Gamma(m)\Gamma(1)\pi}{\Gamma(\frac{3-j}{2}-\frac{C_j}{2}\eps^{N-2})\Gamma(-m-\frac{j-4}{2}-\frac{C_j}{2}\eps^{N-2})\Gamma(m+\frac{j}{2})\Gamma(\frac{j+1}{2})}
\end{aligned}
$$
for $j\ge 2$.

When $j=2L$, we see that 
\begin{equation}\label{eq:int_even_2L}
\int_0^{\pi}P_\nu^{\mu}(\cos\theta)\sin^{m}\theta\, d\theta \simeq \frac{(-1)^{m+L-1}2^{m-2}(m-1)!\pi C_j}{(m+L-1)\Gamma(\frac{3}{2}-L)\Gamma(L+\frac{1}{2})}
\eps^{N-2}
\end{equation}
and when $j=2L-1$, we have
\begin{equation}\label{eq:int_even_2L-1}
\int_0^{\pi}P_\nu^{\mu}(\cos\theta)\sin^{m}\theta\, d\theta \simeq \frac{(-1)^{L-1}2^{m-2}(m-1)!\pi C_j}{L\Gamma(m+L-\frac{1}{2})\Gamma(-m-L+\frac{5}{2})}
\eps^{N-2}.
\end{equation}
When $N=2m-1$, then we have
$$
\mu=-\frac{N-2}{2}=\frac{3}{2}-m, \quad \nu\simeq j+m-\frac{5}{2}+C_j\eps^{N-2}, \quad \alpha=m+\frac{1}{2}.
$$
Then 
$$
\alpha-\nu\simeq3-j-C_j\eps^{N-2},\quad -\mu-\nu+1\simeq 2-j-C_j\eps^{N-2},
$$
$$
\alpha+\nu+1 \simeq 2m+j-1+C_j\eps^{N-2},\quad \nu-\mu+2 \simeq j+2m-2+C_j\eps^{N-2}.
$$
For $j=1$, $-\mu-\nu+1$ does not converge to a nonpositive integer, either. Indeed, we have
$$
\begin{aligned}
\int_0^{\pi}P_\nu^{\mu}(\cos\theta)\sin^{m-1/2}\theta\, d\theta&=\frac{2^{-N/2+1}\pi\Gamma(N)\Gamma(1)}{\Gamma(1)\Gamma (\frac{N+1}{2})\Gamma(\frac{1}{2})\Gamma(N)}+O(\eps^{N-2})\vspace{10pt}\\
&=\frac{2^{-N/2+1}\pi}{\Gamma((N+1)/2)\Gamma(1/2)}+O(\eps^{N-2})\ne0
\end{aligned}
$$
and  for $j\ge 2$, we get 
$$
\int_0^{\pi}P_\nu^{\mu}(\cos\theta)\sin^{m-1/2}\theta\, d\theta\simeq \frac{2^{-(m-3/2)}\Gamma(m-\frac{1}{2})\Gamma(1)\pi }{\Gamma(\frac{3-j}{2}-\frac{C_j}{2}\eps^{N-2})\Gamma(\frac{2-j}{2}-\frac{C_j}{2}\eps^{N-2})\Gamma (m+\frac{j-1}{2})\Gamma(m+\frac{j-2}{2})}
$$
As in the even dimensional case, if $j=2L$, we have
\begin{equation}\label{eq:int_odd_2L}
\int_0^{\pi}P_\nu^{\mu}(\cos\theta)\sin^{m-1/2}\theta\, d\theta\simeq \frac{(-1)^{L}2^{-(m-1/2)}\Gamma(m-\frac{1}{2})(L-1)!\pi }{\Gamma(\frac{3}{2}-L)\Gamma(m+L-\frac{1}{2})(n+L-2)!}C_j\eps^{N-2}
\end{equation}
and 
\begin{equation}\label{eq:int_odd_2L-1}
\int_0^{\pi}P_\nu^{\mu}(\cos\theta)\sin^{m-1/2}\theta\, d\theta\simeq \frac{(-1)^{L-1}2^{-(m-1/2)}\Gamma(m-\frac{1}{2})(L-2)!\pi }{\Gamma(\frac{3}{2}-L)\Gamma(m+L-\frac{3}{2})(n+L-2)!}C_j\eps^{N-2}
\end{equation}
when $j=2L-1.$ 
\par
Thus, for any dimension, if $j=1$, then we have
$$
\int_0^{\pi}P_\nu^{\mu}(\cos\theta)\sin^{\alpha-1}\theta\, d\theta\ne 0.
$$
Also we conclude that 
$$
\int_0^{(1-\eps)\pi}P_\nu^{\mu}(\cos\theta)\sin^{\alpha-1}\theta\, d\theta\ne 0,
$$
when $j=1$ for any small $\eps>0$. 
\par
Moreover, we have the following.
\begin{lemma}\label{le:conv} 
The first eigenfunction $\varphi_1(\theta)$, which is nonnegative and is monotone decreasing,  converges to a constant function as $\eps\to 0$ compact uniformly on $[0, \pi)$. Moreover, there hold
$$
\|\varphi_1\|_{\infty}\to \frac{1}{\sqrt{\omega_N}},\quad 
(1,\varphi_1)\to \sqrt{\omega_N},\quad 
(1,\varphi_1)\|\varphi_1\|_\infty=1+o(1)
$$
as $\eps\to 0$,
where $\omega_N=\frac{2\pi^{(N+1)/2}}{\Gamma((N+1)/2)}$ is the surface area of $\S^N$. 
\end{lemma}
\begin{proof}
As we consider the first eigenfunction, first we note that $\nu\to  (N-2)/2$ and that $\lambda_1\to 0$ as $\eps\to 0$ according to \eqref{eq:nu_ep-n}, \eqref{eq:lambda_n}, \eqref{eq:1steigen}.
Moreover, the first eigenfunction is monotone decreasing in $\theta$ in view of \eqref{eq:rad-form} since $\varphi_1$ is positive due to the characterization of the first eigenfunction.  Then putting $\lambda=\lambda_1$ in \eqref{eq:rad-form}, due to the continuous dependence of a solution 
of an ordinary differential equation on the data, a regular solution to \eqref{eq:rad-form} converges to a regular solution to 
\begin{equation}\label{eq:limit}
\frac{d}{d\theta}(\sin^{N-1}\theta \frac{d}{d\theta}\varphi)=0
\end{equation}
compact uniformly in $[0, \pi)$. \eqref{eq:limit} implies that 
$$
\frac{d}{d\theta}\varphi=\frac{K}{\sin^{N-1}\theta}
$$
with some constant $K$. In order to have a regular solution at $\theta=0$, $K$ must be zero and  we see that a regular solution to \eqref{eq:limit} must be a constant. 

Thus,  as $\eps\to 0$,  
$\varphi_1\to c_N$ for a constant $c_N>0$, compactly uniformly on $[0,\pi)$. 
Hence, we have 
$$(1,\varphi_1)\to c_N\omega_N$$ 
and 
$$\|\varphi_1\|_2^2\to c_N^2\omega_N.$$
Using the normalization assumption
$\|\varphi\|_2=1$ we conclude that $c_N=1/\sqrt{\omega_N}$. Hence
$$(1,\varphi_1)\|\varphi_1\|_\infty=(\sqrt{\omega_N}+o(1))(1/\sqrt{\omega_N}+o(1))=1+o(1),$$
so the assertion follows.
\end{proof}

Only the case $j=1$ is exceptional. Intuitively, we see that the first eigenfunction converges to a positive constant and this makes the limit of $(1,\varphi_1)$ be away from zero as $\eps\to +0$. If $j\ge 2$ then $(1,\varphi_j)\to 0$ as $\eps\to +0$. More precisely, we have the following lemma which will be proved in the next section.

\begin{lemma}\label{prop:1} 
	Let $N\ge 3$. For each $\ell\ge 2$, as $\eps\to 0$ there holds
	\begin{equation}\label{eq:decay}
	(1,\varphi_\ell)= M_{\ell, N} \eps^{N-2}+o(\eps^{N-2})
	\end{equation}
	where 
	$M_{\ell, N}$ is a constant which is independent of $\eps$.
\end{lemma} 

Hence, using Lemmas \ref{le:conv} and \ref{prop:1},  we obtain 
\begin{equation}\label{e-olittleN}
\|w_{\Omega_\eps}\|_{\infty}=\frac{1}{\lambda_1}+O(1),
\end{equation}
$N\ge 4$ exactly as in the case $N=3$ is Section \ref{s32}, and this completes the proof of Proposition~\ref{prop:torsion}.

\section{Proof of Lemma~\ref{prop:1}.} 

Although the integration in \eqref{eq:definite} is done over the interval $[0,\pi]$, we need to evaluate the integration over $[0,(1-\eps)\pi]$.  
Thus, we write 
$$
\int_0^{(1-\eps)\pi}P_\nu^{\mu}(\cos\theta)\sin^{N/2}\theta\, d\theta=\left(\int_0^{\pi}-\int_{(1-\eps)\pi}^{\pi}\right)P_\nu^{\mu}(\cos\theta)\sin^{N/2}\theta\, d\theta.
$$
In order to show that the second term is much smaller than the first term,  we use the asymptotic behaviour of the Legendre functions near $\theta=\pi$. To this end, we use the following formulas.

For this purpose, we regard the expression of the associated Legendre function in terms of the Gauss hypergeometric functions. 
We note that the associated Legendre function is expressed by using the Gauss hypergeometric function as below when $\mu$ is not a positive integer:
\begin{equation}\label{eq:hypergeo}
P_\nu^{\mu}(t)=\frac{1}{\Gamma(1-\mu)}\left(\frac{1+t}{1-t}\right)^{\mu/2}F(-\nu,\nu+1,1-\mu;\frac{1-t}{2}).
\end{equation}
Note that the right-hand side has a singularity at $t=-1$. 
\subsection{Odd dimensional case.}
When $N$ is odd, as we take $\mu=-(N-2)/2$, from \eqref{eq:hypergeo}, we have
$$
P_\nu^{-(N-2)/2}(\cos\theta)=\frac{1}{\Gamma(N/2)}\left(\frac{1+\cos\theta }{1-\cos\theta}\right)^{-(N-2)/4}F(-\nu,\nu+1,\frac{N}{2};\frac{1-\cos\theta}{2}).
$$
From \eqref{eq:conversion} in Section~\ref{sec:gauss}, we have
$$
\begin{aligned}
&F(-\nu,\nu+1,\frac{N}{2};\frac{1-\cos\theta}{2})\vspace{10pt}\\
&=\frac{\Gamma(N/2)\Gamma(\frac{N-2}{2})}{\Gamma(\frac{N}{2}+\nu)\Gamma(\frac{N-2}{2}-\nu)}F(-\nu,\nu+1,2-\frac{N}{2};\frac{1+\cos\theta}{2})
\vspace{10pt}\\
&\hspace{1cm}+\frac{\Gamma(\frac{N}{2})\Gamma(-\frac{N-2}{2})}{\Gamma(-\nu)\Gamma(\nu+1)}\left(\frac{1+\cos\theta}{2}\right)^{(N-2)/2}F(\frac{N}{2}+\nu,\frac{N-2}{2}-\nu,\frac{N}{2};\frac{1+\cos\theta}{2}).
\end{aligned}
$$
In general, there holds $F(a,b,c;x)=1+(ab/c)x+O(x^2)$ near $x=0$ if $a,b,c \not\in\{0,-1,-2,\dots\}$. 
Thus, 
$F(-\nu,\nu+1,N/2;(1-\cos\theta)/2)$ is close to $1$ near $\theta=\pi$. 

In the odd dimensional case ($N=2m-1$), we have seen in the previous section that $\nu=j+m-\frac{5}{2}+C_j\eps^{N-2}+o(\eps^{N-2})$. Thus, according to 
(i) of Lemma~\ref{le:Gamma}, there holds
$$
\frac{1}{\Gamma(\frac{N-2}{2}-\nu)}=\frac{1}{\Gamma(1-j-C_j\eps^{N-2}+o(\eps^{N-2}))}\simeq \eps^{N-2}
$$
and we have 
$$
(1+\cos\theta)^{(N-2)/2}\lesssim \eps^{N-2}
$$ for $\theta\in [(1-\eps)\pi,\pi]$. 
We note that 
$$
\int_{(1-\eps)\pi}^{\pi}P_\nu^{\mu}(\cos\theta)\sin^{N/2}\theta\, d\theta\simeq \int_{(1-\eps)\pi}^{\pi}(1+\cos\theta)^{-(N-2)/4}\sin^{N/2}\theta\, d\theta
$$
holds. Letting $s=\cos\theta$, we have
$$
\begin{aligned}
\int_{(1-\eps)\pi}^{\pi}(1+\cos\theta)^{-(N-2)/4}\sin^{N/2}\theta\, d\theta&\simeq \int_{-1}^{\cos(1-\eps)\pi}(1+s)^{-(N-2)/4}(1-s^2)^{N/4+1/2}\, ds\vspace{10pt}\\
&\simeq \int_{-1}^{\cos(1-\eps)\pi}(1+s)\, ds\simeq \eps^4.
\end{aligned}
$$
Finally, we see that 
$$
\int_{(1-\eps)\pi}^{\pi}P_\nu^{\mu}(\cos\theta)\sin^{N/2}\theta\, d\theta\lesssim \eps^{N+2}.
$$

\subsection{Even dimensional case.}  
Let $N=2m$. Then we have  $\alpha=m-1-\nu$, $\beta=m+\nu$, $\ell=m$. Thus, \eqref{eq:hypergeo} is not applicable and \eqref{eq:P-Fnatu} has some difficulty as $\cos (1-\eps)\pi \to -1$. Thus, in the even dimensional case, the conversion formula  \eqref{eq:conversion} in 
Section~\ref{sec:gauss} does not hold. Instead we use the function $U$ defined as \eqref{eq:U} in 
Section~\ref{sec:gauss}. 
Concerning the relation $U$ and the Gauss hypergeometric function $F$, the conclusion in Beals and Wong~\cite{BW} in p. 276 is 
\begin{equation}\label{eq:inversion}
F(\alpha,\beta,\ell;x)=\Gamma(\ell)U(\alpha,\beta,\alpha+\beta+1-\ell;1-x)
\end{equation}
provided $\alpha+\beta+1-\ell$ is a non-positive integer.  
Also, we use the other expression as 
\begin{equation}\label{eq:inthyergeoU}
P_\nu^{m-1}(\cos\theta)=\frac{\Gamma(m+\nu)(m-1)!}{\Gamma(2+\nu-m)}(\sin^{m-1}\theta) U(m-1-\nu,m+\nu,m;\frac{1+\cos\theta}{2}).
\end{equation}
As $\theta\to \pi-0$, the leading term of $U$ is $(1+\cos\theta)^{1-m}$ when $m\ge 2$ and $\log (1+\cos \theta)$ when $m=1$. 
Also note that $\alpha=m-1-\nu=-j+1-C_j\eps^{N-2}+o(\eps^{N-2})<0$ if $j\ge 2$. Thus, $\Gamma(\alpha)$ in \eqref{eq:U} behaves like $\eps^{N-2}$ order. 
Hence, we see that
$$
\int_{(1-\eps)\pi}^{\pi}P_\nu^{m-1}(\cos\theta)\sin^{m}\theta\, d\theta\simeq \eps^{N-2} \int_{(1-\eps)\pi}^{\pi}(1+\cos\theta)^{1-m}\sin^{2m-1}\theta\, d\theta
$$
holds.
Letting $s=\cos\theta$, for $N=2m$ with $m\ge 2$, we have
$$
\begin{aligned}
\int_{(1-\eps)\pi}^{\pi}(1+\cos\theta)^{1-m}\sin^{2m-1}\theta\, d\theta&\simeq \int_{-1}^{\cos(1-\eps)\pi}(1+s)^{1-m}(1-s^2)^{m}\, ds\vspace{10pt}\\
&\simeq\int_{-1}^{\cos(1-\eps)\pi}(1+s)\, ds\simeq \eps^4.
\end{aligned}
$$
Thus, all the coefficients in Lemma~\ref{prop:1} are determined from \eqref{eq:int_even_2L},  \eqref{eq:int_even_2L-1},  \eqref{eq:int_odd_2L} and  \eqref{eq:int_odd_2L-1} and the proof is now complete.

\section{Properties of the Gauss hypergeometric functions.}\label{sec:gauss}
For the Gauss hypergeometric function $F(a,b,c;x)$, there is a conversion formula, which is useful for the analysis on an odd dimensional case.

Let $a,b,c$ be real numbers such that none of $c$, $a+b+1-c$ and $c+1-a-b$ is a non-positive integer. Then for $x\in (-1,1)$, there holds
\begin{equation}\label{eq:conversion}  
\begin{aligned}
F(a,b,c;x)&=\frac{\Gamma(c)\Gamma(c-a-b)}{\Gamma(c-a)\Gamma(c-b)}F(a,b,a+b+1-c;1-x)
\vspace{10pt}\\
&\hspace{1cm}+\frac{\Gamma(c)\Gamma(a+b-c)}{\Gamma(a)\Gamma(b)}(1-x)^{c-a-b}F(c-a,c-b,1+c-a-b;1-x).
\end{aligned}
\end{equation}
However, this conversion formula is not valid for an even dimensional case. To overcome this difficulty, we introduce 
the other function which is an independent solution to the Gauss hypergeometric differential equation.
The singularity can be controlled according to a device found e.g. in Section 8.4 (pp. 274-- 276) of the book by Beals and Wong \cite{BW} by means of the function $U(\alpha,\beta,\ell;x)$. This function solves the hypergeometric differential equation 
$$
x(1-x)\frac{d^2F}{dx^2}+\{\ell-(\alpha+\beta+1)x\}\frac{dF}{dx}-\alpha\beta F=0
$$
and is linearly independent of $F(\alpha,\beta,\ell;x)$. 
\par
Let $\alpha,\beta$ be non-integer values and $\ell$ be a positive integer.
The function $U(\alpha,\beta,\ell,x)$ is defined as follows
\begin{equation}\label{eq:U}
\begin{array}{l}
U(\alpha,\beta,\ell;x)\vspace{10pt}\\
=\displaystyle{\frac{(-1)^\ell}{\Gamma(\alpha+1-\ell)\Gamma(\beta+1-\ell)(\ell-1)!}\times}\vspace{10pt}\\
\hspace{2mm}\displaystyle{\Bigg[F(\alpha,\beta,\ell;x)\log x}\vspace{10pt}\\
\hspace{4mm}\displaystyle{+\sum_{i=0}^{\infty}\frac{(\alpha)_i(\beta)_i}{(\ell)_i i!}
	\left\{\psi(\alpha+i)+\psi(\beta+i)-\psi(i+1)-\psi(\ell+i)\right\}x^i\Bigg]}
\vspace{10pt}\\
\hspace{6mm}
\displaystyle{+\frac{(\ell-2)!}{\Gamma(\alpha)\Gamma(\beta)}x^{1-\ell}\sum_{i=0}^{\ell-2}\frac{(\alpha+1-\ell)_i (\beta+1-\ell)_i}{(2-\ell)_i i!}x^i},
\end{array}
\end{equation} 
where $\psi(z)$ is the psi (or di-Gamma) function defined as 
$$
\psi(z)=\frac{\Gamma'(z)}{\Gamma(z)},
$$
$(z)_0=1$ and $(z)_i=z(z+1)\cdots(z+i-1)$ with $i\in {\mathbb N}.$ 
The last term in \eqref{eq:U} does not exist if $\ell=0, 1$. 
\par
\section{Properties of the Gamma function}\label{sec:Gamma}
In this subsection, we recall the definition and several properties of the Gamma function.  
One of the definitions of the Gamma function $\Gamma(z)$ is
$$
\Gamma(z)=\lim_{n\to\infty}\frac{(n-1)!n^z}{z(z+1)\cdots (z+n-1)}.
$$
By this definition we see that $\Gamma(z)$ has a pole of order 1 at $z=0,-1,-2,\dots$ and the local behaviour around one of these poles. 
\par 
The definition of the $\psi$-function is 
$$
\psi(z)=\frac{\Gamma' (z)}{\Gamma(z)}=\frac{d}{dz}\log \Gamma(z).
$$
The  $\psi$-function has the following property
\begin{equation}\label{eq:psi-add}
\psi(z+1)=\psi(z)+\frac{1}{z}.
\end{equation} 
The limiting behaviour of the Gamma function is as follows.
\begin{lemma}[cf. Sections 1 and 3 in \cite{MUH}]\label{le:Gamma} Let $n$ and $k$ be nonnegative integers. Then:
	\begin{itemize}
		\item[{\rm (i)}] $\displaystyle{\lim_{\zeta\to 0}\zeta\Gamma(-n-\zeta)}=\displaystyle{\frac{(-1)^{n+1}}{n!}}.$
		\item[{\rm (ii)}] $\displaystyle{\lim_{\zeta\to 0}\zeta\psi(-n-\zeta)}=1.$ 
		\item[{\rm (iii)}] $\displaystyle{\lim_{\zeta\to 0}\frac{\psi(-n-\zeta)}{\Gamma(-k-\zeta)}=(-1)^{k+1}k!}.$
	\end{itemize}
\end{lemma}

\smallskip
\noindent {\bf Acknowledgements.} 
The authors are grateful to the anonymous referee for their comments which helped to improve the exposition in the paper.
Part of this research was carried out while VM was visiting Osaka Prefecture University. 
VM thanks the Department of Mathematical Sciences for its support and hospitality.  YK is supported in part by JSPS KAKENHI Grant Number 19K03588.

\end{document}